\pdfoutput=1
\documentclass[a4paper,11pt]{amsart}

\usepackage[margin=1in,includehead,includefoot]{geometry}
\usepackage[utf8]{inputenc}
\usepackage[T1]{fontenc}
\usepackage{lmodern,microtype}
\usepackage{amsmath,amssymb}
\usepackage{enumitem}
\usepackage{tikz,wrapfig}
\usepackage[unicode]{hyperref}
\usepackage{bookmark}

\hypersetup{colorlinks=true,linkcolor=black,citecolor=black,filecolor=black,urlcolor=black}

\makeatletter
\def\@makefntext{\settowidth\leftskip{\textsuperscript{1}\kern 2pt}\parindent 0pt\leavevmode\llap{\@makefnmark\kern 2pt}}
\let\@adminfootnotes\relax
\let\citationorig\citation
\def\citation#1{\citationorig{#1}\@for\@tempa:=#1\do{\@ifundefined{cit@\@tempa}{\global\@namedef{cit@\@tempa}{}}{}}}
\let\bibitemorig\bibitem
\def\bibitem#1{\@ifundefined{cit@#1}{\typeout{LaTeX Warning: Unused bibitem `#1'}}{}\bibitemorig{#1}}

\def\thmlabel#1{\textup{\bfseries\makebox[1em][l]{#1}}}

\renewenvironment{itemize}{\begin{itemorig}[label=\textbullet, noitemsep, topsep=3pt plus 3pt, labelsep=.6em, leftmargin=1.5em]}{\end{itemorig}}
\newenvironment{enumeratei}{\begin{enumorig}[label=\textup{(\roman*)}, noitemsep, topsep=3pt plus 3pt, leftmargin=*, widest=iii]}{\end{enumorig}}
\newenvironment{enumerate1}{\begin{enumorig}[label=\textup{\arabic*.}, noitemsep, topsep=3pt plus 3pt, leftmargin=1.5em]}{\end{enumorig}}
\newenvironment{enumeratethm}{\begin{enumorig}[label=\thmlabel{\Alph*.}, ref=\textup{\Alph*}, noitemsep, topsep=3pt plus 3pt, leftmargin=*]\itshape}{\end{enumorig}}
\newcounter{theorem}
\newenvironment{theorem}{\begin{enumorig}[label=\thmlabel{\Alph*.}, noitemsep, topsep=\thm@preskip, leftmargin=*]\setcounter{enumi}{\thetheorem}\item\itshape}{\end{enumorig}\stepcounter{theorem}}
\makeatother

\let\leq\leqslant
\let\geq\geqslant
\let\setminus\smallsetminus
\let\Theta\varTheta
\let\Omega\varOmega

\allowdisplaybreaks

\hypersetup{
  pdftitle={Boolean dimension and local dimension},
  pdfauthor={William T. Trotter, Bartosz Walczak},
}

\begin{document}

\title{Boolean dimension and local dimension}
\author[William~T. Trotter\and Bartosz Walczak]{William~T. Trotter\protect\footnotemark\and Bartosz Walczak\protect\footnotemark}

\footnotetext[1]{School of Mathematics, Georgia Institute of Technology, Atlanta, GA, USA, email:\ \href{mailto:trotter@math.gatech.edu}{trotter@math.gatech.edu}}
\footnotetext[2]{Department of Theoretical Computer Science, Faculty of Mathematics and Computer Science, Jagiellonian University, Kraków, Poland, email:\ \href{mailto:walczak@tcs.uj.edu.pl}{walczak@tcs.uj.edu.pl}\\Bartosz Walczak was partially supported by National Science Center of Poland grant 2015/18/E/ST6/00299.}

\begin{abstract}
Dimension is a standard and well-studied measure of complexity of posets.
Recent research has provided many new upper bounds on the dimension for various structurally restricted classes of posets.
Bounded dimension gives a succinct representation of the poset, admitting constant response time for queries of the form ``is $x<y$?''.
This application motivates looking for stronger notions of dimension, possibly leading to succinct representations for more general classes of posets.
We focus on two: \emph{boolean dimension}, introduced in the 1980s and revisited in recent research, and \emph{local dimension}, a very new one.
We determine precisely which values of dimension/boolean dimension/local dimension imply that the two other parameters are bounded.
\end{abstract}

\maketitle

\section{Introduction}

\subsection*{Dimension}

The \emph{dimension} of a poset $P=(X,{\leq})$ is the minimum number of linear extensions of $\leq$ on $X$ the intersection of which gives $\leq$.
More precisely, a \emph{realizer} of a poset $P=(X,{\leq})$ is a set $\{{\leq_1},\ldots,{\leq_d}\}$ of linear extensions of $\leq$ on $X$ such that
\[x\leq y\iff(x\leq_1y)\wedge\cdots\wedge(x\leq_dy),\quad\text{for any }x,y\in X,\]
and the dimension is the minimum size of a realizer.
The concept of dimension was introduced by Dushnik and Miller \cite{DM41} and has been widely studied since.
There are posets with arbitrarily large dimension: the \emph{standard example} $S_k=(\{a_1,\ldots,a_k,b_1,\ldots,b_k\},{\leq})$, where $a_1,\ldots,a_k$ are minimal elements, $b_1,\ldots,b_k$ are maximal elements, and $a_i<b_j$ if and only if $i\neq j$, has dimension $k$ when $k\geq 2$ \cite{DM41}.
On the other hand, the dimension of a poset is at most the width \cite{Hir51}, and it is at most $\frac{n}{2}$ when $n\geq 4$, where $n$ denotes the number of elements \cite{Hir51}.

The \emph{cover graph} of a poset $P=(X,{\leq})$ is the graph on $X$ with edge set $\{xy\colon x<y$ and there is no $z$ with $x<z<y\}$.
A poset is \emph{planar} if its cover graph has a non-crossing \emph{upward drawing} in the plane, which means that every cover graph edge $xy$ with $x<y$ is drawn as a curve that goes monotonically up from $x$ to $y$.
Planar posets that contain a least element and a greatest element are well known to have dimension at most $2$ \cite{BFR72}.
By contrast, spherical posets (i.e.\ posets with upward non-crossing drawings on a sphere) with least and greatest elements can have arbitrarily large dimension \cite{Tro78}.
Trotter and Moore \cite{TM77} proved that planar posets that contain a least element have dimension at most $3$ (and so do posets whose cover graphs are forests) and asked whether all planar
\begin{wrapfigure}[5]{r}{.2\textwidth}
\centering
\vspace*{-1ex}
\begin{tikzpicture}[xscale=.3,yscale=.32,rotate=25]
  \tikzstyle{every node}=[circle,draw,fill,minimum size=3pt,inner sep=0pt]
  \node (a1) at (0.25,-5.5) {};
  \node (a2) at (0.25,-4.5) {};
  \node (a3) at (0.25,-3.5) {};
  \node (a4) at (0.25,-2.5) {};
  \node (a5) at (0.25,-1.5) {};
  \node (b1) at (-0.25,5.5) {};
  \node (b2) at (-0.25,4.5) {};
  \node (b3) at (-0.25,3.5) {};
  \node (b4) at (-0.25,2.5) {};
  \node (b5) at (-0.25,1.5) {};
  \node (z1) at (-5,-0.25) {};
  \node (z2) at (-4,-0.25) {};
  \node (z3) at (-3,-0.25) {};
  \node (z4) at (-2,-0.25) {};
  \node (w1) at (5,0.25) {};
  \node (w2) at (4,0.25) {};
  \node (w3) at (3,0.25) {};
  \node (w4) at (2,0.25) {};
  \path (z1) edge (a1) edge (b2) edge (z2);
  \path (w1) edge (b1) edge (a2) edge (w2);
  \path (z2) edge (a2) edge (b3) edge (z3);
  \path (w2) edge (b2) edge (a3) edge (w3);
  \path (z3) edge (a3) edge (b4) edge (z4);
  \path (w3) edge (b3) edge (a4) edge (w4);
  \path (z4) edge (a4) edge (b5);
  \path (w4) edge (b4) edge (a5);
\end{tikzpicture}
\end{wrapfigure}
posets have bounded dimension.
The answer is no---Kelly \cite{Kel81} constructed planar posets with arbitrarily large dimension (pictured).
Another property of Kelly's posets is that their cover graphs have path-width and tree-width~$3$.
Recent research brought a plethora of new bounds on dimension for structurally restricted posets.
In particular, dimension is bounded for
\begin{itemize}
\item posets with height $2$ and planar cover graphs \cite{FLT10},
\item posets with bounded height and planar cover graphs \cite{ST14},
\item posets with bounded height and cover graphs of bounded tree-width \cite{JMM+16},
\item posets with bounded height and cover graphs excluding a topological minor \cite{Wal17},
\item posets with bounded height and cover graphs of bounded expansion \cite{JMW16},
\item posets with cover graphs of path-width $2$ \cite{BKY16},
\item posets with cover graphs of tree-width $2$ \cite{JMT+},
\item posets with planar cover graphs excluding two incomparable chains of bounded length \cite{HST+}.
\end{itemize}

\subsection*{Boolean dimension}

The \emph{boolean dimension} of a poset $P=(X,{\leq})$ is the minimum number of linear orders on $X$ a boolean combination of which gives $\leq$.
More precisely, a \emph{boolean realizer} of $P$ is a set $\{{\leq_1},\ldots,{\leq_d}\}$ of linear orders on $X$ for which there is a $d$-ary boolean formula $\phi$ such that
\begin{equation}
\label{eq:boolean-realizer}
x\leq y\iff\phi\bigl((x\leq_1y),\ldots,(x\leq_dy)\bigr)\quad\text{for any }x,y\in X,
\end{equation}
and the boolean dimension is the minimum size of a boolean realizer.
The boolean dimension is at most the dimension, because a realizer is a boolean realizer for the formula $\phi(\alpha_1,\ldots,\alpha_d)=\alpha_1\wedge\cdots\wedge\alpha_d$.
Beware that the relation $\leq$ defined by \eqref{eq:boolean-realizer} from arbitrary linear orders ${\leq_1},\ldots,{\leq_d}$ on $X$ and formula $\phi$ is not necessarily a partial order.

Boolean dimension was first considered by Gambosi, Nešetřil, and Talamo \cite{GNT90} and by Nešetřil and Pudlák \cite{NP89}.
The definition above follows \cite{NP89}.
That in \cite{GNT90} allows only formulas $\phi$ of the form $\phi(\alpha_1,\ldots,\alpha_d)=\alpha_i\wedge\psi(\alpha_1,\ldots,\alpha_{i-1},\alpha_{i+1},\ldots,\alpha_d)$ for some $i$.
The purpose of this restriction is unclear---it guarantees antisymmetry but not transitivity of the relation $\leq$ defined by \eqref{eq:boolean-realizer}.
Under that modified definition, it is proved in \cite{GNT90} that boolean dimension $d$ and dimension $d$ are equivalent for $d\in\{1,2,3\}$ (we redo that proof in section \ref{sec:proofs} with no restriction on $\phi$).
The standard examples $S_k$ with $k\geq 4$ have boolean dimension $4$ \cite{GNT90} (see section \ref{sec:proofs}).
Easy counting shows that there are posets on $n$ elements with boolean dimension $\Theta(\log n)$ \cite{NP89}.
This is optimal---every $n$-element poset has boolean dimension $O(\log n)$ witnessed by a formula of length $O(n^2\log n)$ \cite{NP89}.

Nešetřil and Pudlák \cite{NP89} asked whether boolean dimension is bounded for planar posets.
It was proved already in \cite{GNT90} that posets with height $2$ and planar cover graphs have bounded boolean dimension.
Spherical posets with a least element also have bounded boolean dimension \cite{BF96}, contrary to ordinary dimension.
This and the recent progress on dimension of structurally restricted posets have motivated revisiting boolean dimension in current research.

\subsection*{Local dimension}

A \emph{partial linear extension} of a partial order $\leq$ on $X$ is a linear extension of the restriction of $\leq$ to some subset of $X$.
A \emph{local realizer} of $P$ of \emph{width} $d$ is a set $\{{\leq_1},\ldots,{\leq_t}\}$ of partial linear extensions of $\leq$ such that every element of $X$ occurs in at most $d$ of ${\leq_1},\ldots,{\leq_t}$ and
\begin{equation}
\label{eq:local-realizer}
x\leq y\iff\text{there is no }i\in\{1,\ldots,t\}\text{ with }x>_iy,\quad\text{for any }x,y\in X.
\end{equation}
The \emph{local dimension} of $P$ is the minimum width of a local realizer of $P$.
Thus, instead of the size of a local realizer, we bound the number of times any element of $X$ occurs in it.
A set of linear extensions of $\leq$ is a local realizer if and only if it is a realizer.
In particular, the local dimension is at most the dimension.
For arbitrary partial linear extensions ${\leq_1},\ldots,{\leq_t}$ of $\leq$ on subsets of $X$, the relation $\leq$ defined by \eqref{eq:local-realizer} is not necessarily a partial order---it may fail to be antisymmetric or transitive.
It is antisymmetric, for example, if one of ${\leq_1},\ldots,{\leq_t}$ is a linear extension of $\leq$ on $X$.

The concept of local dimension was proposed very recently by Ueckerdt \cite{Uec} and originates from concepts studied in \cite{BSU,KU16}.
Ueckerdt \cite{Uec} also noticed that the standard examples $S_k$ with $k\geq 3$ have boolean dimension $3$.

\subsection*{Results}

Extending the results on boolean dimension from \cite{GNT90}, for each $d$, we determine whether posets with dimension/boolean dimension/local dimension $d$ have the other two parameters bounded or unbounded.
Here is the full picture:
\begin{enumeratethm}
\item\label{item:A} Boolean dimension\/ $d$ and dimension\/ $d$ are equivalent for\/ $d\in\{1,2,3\}$ \cite{GNT90}.
\item\label{item:B} Local dimension\/ $d$ and dimension\/ $d$ are equivalent for\/ $d\in\{1,2\}$.
\item\label{item:C} The standard examples\/ $S_k$ have boolean dimension\/ $4$ when\/ $k\geq 4$ \cite{GNT90}, local dimension\/ $3$ when\/ $k\geq 3$ \cite{Uec}, and dimension\/ $k$ when\/ $k\geq 2$ \cite{DM41}.
\item\label{item:D} There are posets with boolean dimension\/ $4$ and unbounded local dimension.
\item\label{item:E} Posets with local dimension\/ $3$ have bounded boolean dimension.
\item\label{item:F} There are posets with local dimension\/ $4$ and unbounded boolean dimension.
\end{enumeratethm}
We present proofs of \ref{item:A}--\ref{item:F} in the next section.

Other new results concern boolean dimension and local dimension of structurally restricted posets.
In particular, posets with cover graphs of bounded path-width have bounded boolean dimension \cite{MW} and bounded local dimension \cite{BPS+}, while local dimension is unbounded for posets with cover graphs of tree-width $3$ \cite{BPS+,BGT} and for planar posets \cite{BGT}.
It remains open whether boolean dimension is bounded for posets with cover graphs of bounded tree-width (in particular, tree-width $3$) and for planar posets.
There are $n$-element posets with local dimension $\Theta(\sqrt{n})$ \cite{BPS+}, and conceivably the right bound is $(\frac{1}{2}-o(1))n$.

\section{Proofs}
\label{sec:proofs}

\begin{theorem}
Boolean dimension\/ $d$ and dimension\/ $d$ are equivalent for\/ $d\in\{1,2,3\}$.
\end{theorem}

\begin{proof}
We basically repeat the argument given in \cite{GNT90} but avoiding the restriction on functions $\phi$ imposed therein.
Let $P=(X,{\leq})$ be a poset with boolean dimension $d$ and $\{{\leq_1},\ldots,{\leq_d}\}$ be its boolean realizer for a formula $\phi$.
Reflexivity of $\leq$ implies $\phi(1,\ldots,1)=1$.
Without loss of generality, assume $\phi(\alpha)=0$ when $\phi(\alpha)$ is never used by \eqref{eq:boolean-realizer}.
This and antisymmetry of $\leq$ imply $\phi(\alpha)=0$ or $\phi(\overline{\alpha})=0$ for every $\alpha\in\{0,1\}^d$.
In particular, $\phi(\alpha)=1$ for at most half of the tuples $\alpha$.

If $d=1$, then $\phi(1)=1$ and $\phi(0)=0$, so $\{\leq_1\}$ is a realizer of $P$.
This shows that boolean dimension $1$ and dimension $1$ are equivalent.

Now, let $d\in\{2,3\}$.
If $\phi(\alpha)=\phi(\overline{\alpha})=0$ for at most one pair $\alpha,\overline{\alpha}$ (as it is for $d=2$), then the strict partial order $\prec$ on $X$ defined by $x\prec y\iff\bigl((x<_1y),\ldots,(x<_dy)\bigr)=\alpha$ for distinct $x,y\in X$ is a transitive orientation of the incomparability graph of $P$, so $P$ has dimension at most $2$ \cite[Theorem 3.61]{DM41}.
This shows that boolean dimension $2$ and dimension $2$ are equivalent.

For $d=3$, to complete the proof that boolean dimension $3$ and dimension $3$ are equivalent, we consider the three cases (up to symmetry) in which $\phi(\alpha)=1$ for at most two tuples $\alpha$.
\begin{enumerate1}
\item If $\phi(\alpha)=1$ for $\alpha=(1,1,1)$ only, then $\{{\leq_1},{\leq_2},{\leq_3}\}$ is a realizer of $P$.
\item If $\phi(\alpha)=1$ for $\alpha\in\{(1,1,0),(1,1,1)\}$ only, then $\{{\leq_1},{\leq_2}\}$ is a realizer of $P$.
\item\leavevmode\hbox{If $\phi(\alpha)=1$ for $\alpha\in\{(0,0,1),(1,1,1)\}$ only, then} the strict partial order $\prec$ on $X$ defined by $x\prec y\iff\bigl((x<_1y)\wedge(x>_2y)\bigr)$ for distinct $x,y\in X$ is a transitive orientation of the incomparability graph of $P$, so $P$ has dimension at most $2$ \cite[Theorem 3.61]{DM41}.\qedhere
\end{enumerate1}
\end{proof}

\begin{theorem}
Local dimension\/ $d$ and dimension\/ $d$ are equivalent for\/ $d\in\{1,2\}$.
\end{theorem}

\begin{proof}
If a poset $P=(X,{\leq})$ has local dimension $1$, then a local realizer of $P$ of width $1$ must consist of a single full linear order on $X$, because antisymmetry of $\leq$ requires that every pair $x,y\in X$ occurs in at least one partial linear extension.

Now, let $P=(X,{\leq})$ be a poset with local dimension $2$, and consider a local realizer of $P$ of width $2$.
If $x,y\in X$ are incomparable in $\leq$, then both occurrences of $x$ and $y$ are in the same two partial linear extensions, where $x<y$ in one and $x>y$ in the other.
Therefore, the subposet of $P$ induced on every connected component $C$ of the incomparability graph of $P$ is witnessed by two partial linear extensions, which restricted to $C$ form a realizer of that subposet.
These realizers stacked according to the order $\leq$ form a realizer of $P$ of size $2$.
\end{proof}

\begin{theorem}
The standard examples\/ $S_k$ have boolean dimension\/ $4$ when\/ $k\geq 4$, local dimension\/ $3$ when\/ $k\geq 3$, and dimension\/ $k$ when\/ $k\geq 2$.
\end{theorem}

\begin{proof}
It was observed in \cite{GNT90} that the standard example $S_k$ has boolean dimension $4$ (when $k\geq 4$), witnessed by the formula $\phi(\alpha)=\alpha_1\wedge\alpha_2\wedge(\alpha_3\vee\alpha_4)$ and the following four linear orders:
\begin{alignat*}{8}
&a_1&&<\cdots<a_k&&<b_1&&<\cdots<b_k,\qquad\qquad &&b_1&&<a_1&&<\cdots<b_k&&<a_k,\\
&a_k&&<\cdots<a_1&&<b_k&&<\cdots<b_1,\qquad\qquad &&b_k&&<a_k&&<\cdots<b_1&&<a_1.
\end{alignat*}

Ueckerdt \cite{Uec} observed that $S_k$ has local dimension $3$ (when $k\geq 3$), witnessed by the two linear extensions above on the left and $k$ partial linear extensions each of the form $b_i<a_i$.

Only one pair $a_i,b_i$ can be ordered as $b_i<a_i$ in a single linear extension, so the dimension of $S_k$ is at least $k$.
A realizer of size exactly $k$ can be constructed easily when $k\geq 2$, see \cite{DM41}.
\end{proof}

\begin{theorem}
There are posets with boolean dimension\/ $4$ and unbounded local dimension.
\end{theorem}

\begin{proof}
Another well-known construction of posets with arbitrarily large dimension involves incidence posets of complete graphs: $P_n=(V\cup E,{\leq})$, where $V=\{v_1,\ldots,v_n\}$ are the minimal elements, $E=\{v_1v_2,v_1v_3,\ldots,v_{n-1}v_n\}$ are the maximal elements, and the only comparable pairs are $v_i<v_iv_j$ and $v_j<v_iv_j$ for $i\neq j$.
The dimension of $P_n$ is at least $\log_2\log_2n$ \cite[Theorem 4.22]{DM41}.
The boolean dimension of $P_n$ is at most $4$, witnessed by the formula $\phi(\alpha)=(\alpha_1\wedge\alpha_2)\vee(\alpha_3\wedge\alpha_4)$ and the following four linear orders:
\begin{alignat*}{3}
A_1&<\cdots<A_n&,\enspace &\text{where each }A_i &&\text{ has form}\enspace v_i<v_iv_{i+1}<\cdots<v_iv_n,\\
B_n&<\cdots<B_1&,\enspace &\text{where each }B_i &&\text{ has form}\enspace v_i<v_iv_n<\cdots<v_iv_{i+1},\\
C_1&<\cdots<C_n&,\enspace &\text{where each }C_i &&\text{ has form}\enspace v_i<v_1v_i<\cdots<v_{i-1}v_i,\\
D_n&<\cdots<D_1&,\enspace &\text{where each }D_i &&\text{ has form}\enspace v_i<v_{i-1}v_i<\cdots<v_1v_i.
\end{alignat*}
The local dimension of $P_n$ is unbounded as $n\to\infty$.
For suppose $P_n$ has a local realizer of width $d$.
Enumerate the occurrences of each element of $V\cup E$ in the local realizer from $1$ to (at most)~$d$.
Each triple $v_iv_jv_k$ ($i<j<k$) can be assigned a color $(p,q)$ so that $v_iv_k<v_j$ in a partial linear extension containing the $p$th occurrence of $v_j$ and the $q$th occurrence of $v_iv_k$.
By Ramsey's theorem, if $n$ is large enough compared to $d$, then there is a quadruple $v_iv_jv_kv_\ell$ ($i<j<k<\ell$) with all four triples of the same color $(p,q)$.
It follows that the $p$th occurrences of $v_j$ and $v_k$ and the $q$th occurrences of $v_iv_\ell$, $v_iv_k$, and $v_jv_\ell$ are all in the same partial linear extension, which therefore contains a cycle $v_j<v_jv_\ell<v_k<v_iv_k<v_j$, a contradiction.
\end{proof}

\begin{theorem}
Posets with local dimension\/ $3$ have bounded boolean dimension.
\end{theorem}

\begin{proof}
Let $P=(X,{\leq})$ be a poset with a local realizer of width $3$ consisting of partial linear extensions that we call \emph{gadgets}.
We construct a boolean realizer $\{{\leq^\star},{\leq_1},{\leq_1'},\ldots,{\leq_d},{\leq_{\smash[t]{d}}'}\}$ for a formula of the form $\alpha^\star\wedge(\alpha_1\vee\alpha_1')\wedge\cdots\wedge(\alpha_d\vee\alpha_{\smash[t]{d}}')$.
The order $\leq^\star$ is an arbitrary linear extension of $\leq$ on $X$.
Each pair of orders ${\leq_i},{\leq_i'}$ is defined by $X_1<_i\cdots<_iX_t$ and $X_t<_i'\cdots<_i'X_1$, where $\{X_1,\ldots,X_t\}$ is some partition of $X$ into \emph{blocks} such that every block $X_j$ is completely ordered by some gadget and that order is inherited by $\leq_i$ and $\leq_i'$.
Then, we have $x<y$ for the relation $\leq$ defined by \eqref{eq:boolean-realizer} if and only if $x<^\star y$ and $x<y$ in every block containing both $x$ and $y$.
It remains to construct a bounded number of partitions of $X$ into blocks so that for any $x,y\in X$, if $x<^\star y$ and $x>y$ in some gadget, then $x>y$ in some block in at least one of the partitions.

Without loss of generality, assume that each element $x\in X$ has exactly $3$ occurrences in the gadgets---enumerate them as $x^1,x^2,x^3$ according to a fixed order of the gadgets.
For each $p\in\{1,2,3\}$, form a partition of $X$ by restricting every gadget to elements of the form $x^p$.
These three partitions witness all comparabilities of the form $x^p>y^p$ within gadgets.

Now, let $G$ be a graph on $X$ where $xy$ (with $x<^\star y$) is an edge if and only if $x^p>y^q$ ($p\neq q$) in some gadget.
Thus $G$ is a subgraph of the incomparability graph of $P$.
Suppose $\chi(G)>38$.
It follows that $G$ has an edge $uv$ such that $\chi(G[X_{uv}])\geq 19$, where $X_{uv}=\{x\in X\colon u<^\star x<^\star v\}$ \cite[Lemma 2.1]{McG96}.
Let $X_u=\{x\in X_{uv}\colon u\not\leq x\}$ and $X_v=\{x\in X_{uv}\colon x\not\leq v\}$.
Thus $X_{uv}=X_u\cup X_v$, as $u\not\leq v$.
Let a color of $x\in X_u$ be a quadruple $(p,q,r,s)$ with $p<q$ and $r<s$, where either $x^p<u^q$ and $x^r>u^s$ or $x^p>u^q$ and $x^r<u^s$ in some gadgets.
There are $9$ possible colors (quadruples).
The coloring of $G[X_u]$ thus obtained is proper---whenever $x,y\in X_u$ have the same color, $x^p$ and $y^p$ are in the same gadget, as well as $x^r$ and $y^r$ are in the same gadget; this contradicts the fact that the edge $xy$ of $G$ is witnessed by some $x^i$ and $y^j$ with $i\neq j$ occurring in the same gadget (this is where we use the bound $3$ on the number of occurrences).
Thus $\chi(G[X_u])\leq 9$, and similarly $\chi(G[X_v])\leq 9$, which yields $\chi(G[X_{uv}])\leq 18$.
This contradiction shows $\chi(G)\leq 38$.

Let $c$ be a proper $38$-coloring of $G$.
For $1\leq p<q\leq 3$ and any distinct colors $a,b$, form a partition of $X$ by restricting every gadget to elements of the form $x^p$ with $c(x)=a$ and $y^q$ with $c(y)=b$ (adding singletons if necessary to obtain a full partition of $X$).
The $4218$ partitions thus obtained have the desired property.
The resulting boolean realizer of $P$ has size $8443$.
\end{proof}

\begin{theorem}
There are posets with local dimension\/ $4$ and unbounded boolean dimension.
\end{theorem}

\begin{proof}
When $(V,E)$ is an acyclic digraph, $v\in V$, and $X,Y\subseteq V$, let $E(X,v)=\{xv\in E\colon x\in X\}$, $E(v,Y)=\{vy\in E\colon y\in Y\}$, and $E(X,Y)=\{xy\in E\colon x\in X$ and $y\in Y\}$ ($xy$ denotes a directed edge from $x$ to $y$).
For every $k\geq 1$, we construct an acyclic digraph $G=(V,E)$ with $\chi(G)>k$, a poset $P=(E,{\leq})$, and a local realizer $\{{\leq_A},{\leq_B}\}\cup\{{\leq_v}\colon v\in V\}$ of $P$ of width $4$, where
\begin{enumeratei}
\item\label{item:i} $\leq_A$ is a linear extension of $\leq$ on $E$ such that $E(V,v)<_AE(v,V)$ for every $v\in V$,
\item\label{item:ii} $\leq_B$ is a linear extension of $\leq$ on $E$ such that $E(V,v)<_BE(v,V)$ and $E(v,V)$ occurs as a contiguous block in $\leq_B$ for every $v\in V$,
\item\label{item:iii} $\leq_v$ is a \emph{gadget}---a partial linear extension of the form $E(v,V)<_vE(V,v)$ for every $v\in V$.
\end{enumeratei}

The construction is an adaptation of the well-known construction of triangle-free graphs with arbitrarily large chromatic number from \cite{Des54,KK54}.
For $k=1$, let $V=\{u,v\}$, $E=\{uv\}$, and ${\leq_A},{\leq_B},{\leq_u},{\leq_v}$ be trivial orders on $E$.
Now, suppose that $k\geq 2$ and the construction can be performed for $k-1$.
Let $r$ be the number of vertices in that construction, $s=k(r-1)+1$, and $n=\binom{s}{r}$.
For $1\leq i\leq n$, let $G^i=(V^i,E^i)$, $P^i=(E^i,{\leq^i})$, and $\{{\leq^i_A},{\leq^i_B}\}\cup\{{\leq^i_v}\colon v\in V^i\}$ be separate instances of the construction for $k-1$.
Let $X=\{x_1,\ldots,x_s\}$ be yet a separate set of $s$ vertices.
Let $X^1,\ldots,X^n$ be the $r$-element subsets of $X$.
Let $V=X\cup V^1\cup\cdots\cup V^n$ and $E=\bigcup_{i=1}^n\bigl(\{x^i_1v^i_1,\ldots,x^i_rv^i_r\}\cup E^i\bigr)$, where
\begin{itemize}
\item $x^i_1,\ldots,x^i_r$ are the vertices in $X^i$ in the same order as in the sequence $x_1,\ldots,x_s$,
\item $v^i_1,\ldots,v^i_r$ are the vertices in $V^i$ ordered so that $E^i(v^i_1,V^i)<^i_B\cdots<^i_BE^i(v^i_r,V^i)$.
\end{itemize}
Let $G=(V,E)$.
Clearly, $G$ is an acyclic digraph.
The assumption that $\chi(G^i)>k-1$ for all~$i$ implies $\chi(G)>k$ \cite{Des54,KK54}.
Indeed, in any proper $k$-coloring of $G$, at least one of the sets $X^i$ would be monochromatic, which would yield $\chi(G^i)\leq k-1$, a contradiction.
For $1\leq j\leq s$, let $N_j=\{v\in V\colon x_jv\in E\}$.
Let $\leq_A$ and $\leq_B$ be linear orders on $V$ such that
\begin{itemize}
\item $E(X^1,V^1)<_AE^1<_A\cdots<_AE(X^n,V^n)<_AE^n$ and the restriction of $\leq_A$ to each $E^i$ is $\leq^i_A$,
\item $E(x_1,V)<_BE(N_1,V)<_B\cdots<_BE(x_s,V)<_BE(N_s,V)$.
\end{itemize}
The latter property implies that the restriction of $\leq_B$ to each $E^i$ is $\leq^i_B$.
Finally, for every $x\in X$, let $\leq_x$ be a new gadget on $E(x,V)$, and for $v\in V^i$ and $1\leq i\leq n$, let $\leq_v$ be $\leq^i_v$ with $E(X,v)$ (which is just one edge) added on top.
This guarantees properties \ref{item:i}--\ref{item:iii}.
Let $\leq$ be the relation on $E$ defined from $\{{\leq_A},{\leq_B}\}\cup\{{\leq_v}\colon v\in V\}$ by \eqref{eq:local-realizer}.
It follows that the restriction of $\leq$ to each $E^i$ is $\leq^i$.
It remains to show that the relation $\leq$ is a partial order, so that $P=(E,{\leq})$ is a poset and $\{{\leq_A},{\leq_B}\}\cup\{{\leq_v}\colon v\in V\}$ is its local realizer.

Reflexivity and antisymmetry of $\leq$ are clear.
For transitivity, suppose $e,f,g\in E$, $e<f$, and $f<g$, but $e\not\leq g$.
The assumption that $e<f$ and $f<g$ implies $e<_Af<_Ag$ and $e<_Bf<_Bg$.
Since $e\not\leq g$, the edges $e$ and $g$ must occur as $e>g$ in some gadget.
We consider four cases.
\begin{enumerate1}
\item If $e,g\in E^i$ for some $i$, then the definition of $\leq_A$ implies $f\in E^i$, so $e<^if<^ig$.
This and the assumption that $e\not\leq g$ contradict the fact that $\leq^i$ is the restriction of $\leq$ to $E^i$.
\item If $e,g\in E(x,V)$ for some $x\in X$, then the definition of $\leq_B$ implies $f\in E(x,V)$.
This yields $e<_xf<_xg$, and the only gadget containing both $e$ and $g$ fails to witness $e\not\leq g$.
\item If $e=x_jv$ and $g=uv$ for some $x_j\in X$ and $u,v\in V\setminus X$, then $u,v\in V^i$ for some $i$.
The definition of $E$ implies that there is an edge $x_{j'}u\in E$, where $j'<j$, and therefore $g\in E(N_{j'},V)<_BE(x_j,V)\ni e$, a contradiction.
\item If $e=x_jv$ and $g=vw$ for some $x_j\in X$ and $v,w\in V\setminus X$, then $v,w\in V^i$ for some~$i$.
The~definitions of $\leq_A$ and $\leq_B$ imply
\begin{equation*}
f\in\bigl(E(X^i,V^i)\cup E^i\bigr)\cap\bigl(E(x_j,V)\cup E(N_j,V)\bigr)=\{x_jv\}\cup E(v,V^i).
\end{equation*}
This yields $e<_vf<_vg$, and the only gadget containing both $e$ and $g$ fails to witness $e\not\leq g$.
\end{enumerate1}
This shows that $\leq$ is transitive, thus completing the proof of correctness of the construction.

Let $k=2^{2^{2^d}}$.
We show that the poset $P$ resulting from the construction above has boolean dimension greater than $d$.
For suppose $\{{\leq_1},\ldots,{\leq_d}\}$ is a boolean realizer of $P$ for a formula~$\phi$.
Let $G'=(E,A)$ be the \emph{arc digraph} of $G$, and let $G''=(A,B)$ be the \emph{arc digraph} of $G'$.
That is, $A=\{uvw\colon uv,vw\in E\}$ and $B=\{uvwx\colon uvw,vwx\in A\}=\{uvwx\colon uv,vw,wx\in E\}$.
It follows that $\chi(G')\geq\log_2\chi(G)$ and $\chi(G'')\geq\log_2\chi(G')$ \cite[Theorem 9]{HE72}, and thus $\chi(G'')>2^d$.
For $uvw\in A$, let $\alpha(uvw)=\bigl((uv<_1vw),\ldots,(uv<_dvw)\bigr)\in\{0,1\}^d$; the fact that $uv>_vvw$ implies $uv\not\leq vw$ and thus $\phi(\alpha(uvw))=0$.
Let $uvwx\in B$.
We have $uv<_Avw<_Awx$ and $uv<_Bvw<_Bwx$, which implies $uv<wx$, because no gadget contains both $uv$ and $wx$.
If $\alpha(uvw)=\alpha(vwx)=\alpha$, then transitivity of ${\leq_1},\ldots,{\leq_d}$ implies $\bigl((uv<_1wx),\ldots,(uv<_dwx)\bigr)=\alpha$.
This, $\phi(\alpha)=0$, and $uv<wx$ result in a contradiction.
Therefore, $\alpha\colon A\to\{0,1\}^d$ is a proper $2^d$-coloring of $G''$.
This contradicts the fact that $\chi(G'')>2^d$.
\end{proof}


\begin{thebibliography}{99}

\bibitem{BFR72}
\href{http://doi.org/10.1002/net.3230020103}{Kirby~A. Baker, Peter~C. Fishburn, and Fred~S. Roberts, Partial orders of dimension $2$, \emph{Networks} 2~(1), 11--28, 1972}.

\bibitem{BPS+}
Fidel Barrera Cruz, Thomas Prag, Heather~C. Smith, Libby Taylor, and William~T. Trotter, Local dimension, dimension and topological graph theory, manuscript.

\bibitem{BKY16}
\href{http://doi.org/10.1007/s11083-015-9359-7}{Csaba Biró, Mitchel~T. Keller, and Stephen~J. Young, Posets with cover graph of pathwidth two have bounded dimension, \emph{Order} 33~(2), 195--212, 2016}.

\bibitem{BSU}
\href{http://arxiv.org/abs/1609.09447}{Thomas Bläsius, Peter Stumpf, and Torsten Ueckerdt, Local and union boxicity, arXiv:1609.09447}.

\bibitem{BGT}
Bartłomiej Bosek, Jarosław Grytczuk, and William~T. Trotter, Local dimension is unbounded for planar posets, manuscript.

\bibitem{BF96}
\href{http://doi.org/10.1007/BF00338743}{Graham~R. Brightwell and Paolo~G. Franciosa, On the Boolean dimension of spherical orders, \emph{Order} 13~(3), 233--243, 1996}.

\bibitem{Des54}
\href{http://doi.org/10.2307/2307489}{Blanche Descartes, Solution to advanced problem 4526, \emph{Amer. Math. Monthly} 61~(5), 352--353, 1954}.

\bibitem{DM41}
\href{http://doi.org/10.2307/2371374}{Ben Dushnik and Edwin~W. Miller, Partially ordered sets, \emph{Amer. J. Math.} 63~(3), 600--610, 1941}.

\bibitem{FLT10}
\href{http://doi.org/10.1016/j.disc.2009.11.005}{Stefan Felsner, Ching Man Li, and William~T. Trotter, Adjacency posets of planar graphs, \emph{Discrete Math.} 310~(5), 1097--1104, 2010}.

\bibitem{GNT90}
\href{http://doi.org/10.1016/0304-3975(90)90125-2}{Giorgio Gambosi, Jaroslav Nešetřil, and Maurizio Talamo, On locally presented posets, \emph{Theor. Comput. Sci.} 70~(2), 251--260, 1990}.

\bibitem{HE72}
\href{http://doi.org/10.1016/0095-8956(72)90057-3}{Charles~C. Harner and Roger~C. Entringer, Arc colorings of digraphs, \emph{J. Combin. Theory Ser.~B} 13~(3), 219--225, 1972}.

\bibitem{Hir51}
\href{http://hdl.handle.net/2297/33696}{Toshio Hiraguchi, On the dimension of partially ordered sets, \emph{Sci. Rep. Kanazawa Univ.} 1~(2), 77--94, 1951}.

\bibitem{HST+}
\href{http://arxiv.org/abs/1608.08843}{David~M. Howard, Noah Streib, William~T. Trotter, Bartosz Walczak, and Ruidong Wang, The dimension of posets with planar cover graphs excluding two long incomparable chains, arXiv:1608.08843}.

\bibitem{JMM+16}
\href{http://doi.org/10.1007/s00493-014-3081-8}{Gwenaël Joret, Piotr Micek, Kevin~G. Milans, William~T. Trotter, Bartosz Walczak, and Ruidong Wang, Tree-width and dimension, \emph{Combinatorica} 36~(4), 431--450, 2016}.

\bibitem{JMT+}
\href{http://doi.org/10.1007/s11083-016-9395-y}{Gwenaël Joret, Piotr Micek, William~T. Trotter, Ruidong Wang, and Veit Wiechert, On the dimension of posets with cover graphs of treewidth $2$, \emph{Order}, in press}.

\bibitem{JMW16}
\href{http://doi.org/10.1137/1.9781611974331.ch125}{Gwenaël Joret, Piotr Micek, and Veit Wiechert, Sparsity and dimension, in: \emph{27th Annual ACM-SIAM Symposium on Discrete Algorithms (SODA 2016)}, pp.~1804--1813, SIAM, Philadelphia, 2016}.

\bibitem{Kel81}
\href{http://doi.org/10.1016/0012-365X(81)90203-X}{David Kelly, On the dimension of partially ordered sets, \emph{Discrete Math.} 35~(1\nobreakdash--3), 135--156, 1981}.

\bibitem{KK54}
\href{http://doi.org/10.2307/2372652}{John~B. Kelly and Le Roy~M. Kelly, Path and circuits in critical graphs, \emph{Amer. J. Math.} 76~(4), 786--792, 1954}.

\bibitem{KU16}
\href{http://doi.org/10.1016/j.disc.2015.10.023}{Kolja Knauer and Torsten Ueckerdt, Three ways to cover a graph, \emph{Discrete Math.} 339~(2), 745--758, 2016}.

\bibitem{McG96}
\href{http://doi.org/10.1016/0012-365X(95)00316-O}{Sean McGuinness, On bounding the chromatic number of L-graphs, \emph{Discrete Math.} 154~(1\nobreakdash--3), 179--187, 1996}.

\bibitem{MW}
Piotr Micek and Bartosz Walczak, unpublished result.

\bibitem{NP89}
\href{http://doi.org/10.1007/978-3-642-61324-1_12}{Jaroslav Nešetřil and Pavel Pudlák, A note on boolean dimension of posets, in: Gábor Halász and Vera T.~Sós (eds.), \emph{Irregularities of Partitions}, vol.~8 of \emph{Algorithms Combin.}, pp.~137--140, Springer, Berlin, 1989}.

\bibitem{ST14}
\href{http://doi.org/10.1016/j.ejc.2013.06.017}{Noah Streib and William~T. Trotter, Dimension and height for posets with planar cover graphs, \emph{European J. Combin.} 35, 474--489, 2014}.

\bibitem{Tro78}
\href{http://doi.org/10.1007/bfb0070411}{William~T. Trotter, Order preserving embeddings of aographs, in: Yousef Alavi and Don~R. Lick (eds.), \emph{Theory and Applications of Graphs}, vol.~642 of \emph{Lecture Notes Math.}, pp.~572--579, Springer, Berlin, 1978}.

\bibitem{TM77}
\href{http://doi.org/10.1016/0095-8956(77)90048-X}{William~T. Trotter and John~I. Moore, The dimension of planar posets, \emph{J. Combin. Theory Ser.~B} 22~(1), 54--67, 1977}.

\bibitem{Uec}
Torsten Ueckerdt, Order \& Geometry Workshop, Gułtowy, 2016.

\bibitem{Wal17}
\href{http://doi.org/10.1016/j.jctb.2016.09.001}{Bartosz Walczak, Minors and dimension, \emph{J. Combin. Theory Ser.~B} 122, 668--689, 2017}.

\end{thebibliography}
\end{document}